\documentclass[a4paper,12pt,twoside]{amsart} 
\usepackage{amsmath,amsthm}
\usepackage{amsfonts}
\usepackage{amssymb}
\usepackage[english]{babel}

\newtheorem{theo}{Theorem}[section]
\newtheorem{lem}[theo]{Lemma}
\newtheorem{prop}[theo]{Proposition}
\newtheorem{cor}[theo]{Corollary}

\renewcommand{\P}{{\mathbb P}}

\newcommand{\e}{{\text{e}}}

\newcommand{\N}{{\mathbb N}}
\newcommand{\Z}{{\mathbb Z}}
\newcommand{\R}{{\mathbb R}}
\newcommand{\T}{{\mathbb T}}

\newcommand{\ba}{{\bf a}}
\newcommand{\aveN}{{\frac1N\sum_{k=1}^N}}
\newcommand{\limaveN}{\lim_{N\to\infty} \aveN}

\theoremstyle{remark}
\newtheorem*{ex}{\bf Example}
\newtheorem*{remark}{\bf Remark}

\begin{document}

\baselineskip=14pt

\title{On modulated ergodic theorems}

\author{Tanja Eisner}
\address{Institute of Mathematics, University of Leipzig, Germany}
\email{tatjana.eisner@math.uni-leipzig.de}

\author{Michael Lin}
\address{Department of Mathematics, Ben-Gurion University, Beer-Sheva, Israel}
\email{lin@math.bgu.ac.il}


\begin{abstract}
Let $T$ be a weakly almost periodic (WAP) linear operator on a Banach space $X$. 
A sequence of scalars $(a_n)_{n\ge 1}$ {\it modulates}  $T$ on  $Y \subset X$ 
if $\frac1n\sum_{k=1}^n a_kT^k x$ converges in norm for every $x \in Y$. 
We obtain a sufficient condition for $(a_n)$ to modulate every WAP operator on the space of its
flight vectors, a necessary and sufficient condition for (weakly) modulating every WAP operator 
$T$ on the space of its (weakly) stable vectors, and sufficient conditions for modulating
every contraction on a Hilbert space on the space of its weakly stable vectors. 
We study as an example modulation by the modified von Mangoldt function 
$\Lambda'(n):=\log n1_\mathbb P(n)$ (where $\mathbb P =(p_k)_{k\ge 1}$ is the sequence of primes),
and show that, as in the scalar case, convergence of the corresponding modulated averages is 
equivalent to convergence of the averages along the primes $\frac1n\sum_{k=1}^n T^{p_k}x$. 
We then prove that for any contraction  $T$ on a Hilbert space $H$ and $x \in H$, 
and also for every invertible $T$ with $\sup_{n \in \mathbb Z} \|T^n\| <\infty$ on 
$L^r(\Omega,\mu)$ ($1<r< \infty$) and $f \in L^r$, the averages along the primes converge.

\end{abstract}

\dedicatory{Dedicated to Simeon Reich on the occasion of his 70th birthday}

\maketitle

\section{Introduction}

Let $T$ be a power-bounded  (linear) operator on a (real or complex) reflexive Banach space $X$.
The mean ergodic theorem, proved independently by  Lorch, Yosida and Kakutani, says that for
every $x \in X$ the averages $\frac1n \sum_{k=1}^n T^kx$ converge in norm.
\medskip

{\bf Definition.} A bounded linear operator $T$ on a Banach space $X$ is called {\it (weakly)
almost periodic} if for every $x \in X$ the orbit $\{T^nx\}_{n\ge 0}$ is (weakly) conditionally 
compact. A weakly almost periodic operator is necessarily power-bounded. Every power-bounded $T$ on 
a reflexive Banach space is weakly almost periodic.
\medskip

The statement of the mean ergodic theorem by Kakutani and by Yosida, and their proofs, show the
norm convergence of the averages $\frac1n \sum_{k=1}^n T^kx$ for every $x \in X$ when $T$
is weakly almost periodic. The convergence induces a decomposition of the space ({\it the 
ergodic decomposition})
$$
X=F(T)\oplus \overline{(I-T)X}
$$
where $F(T)$ is the space of fixed points of $T$. The limit operator $E(T)$ is the projection onto
$F(T)$ corresponding to the ergodic decomposition.
\medskip

In this article we are interested in sequences of scalars $\ba = (a_k)_{k>0}$ which yield 
{\it modulated norm ergodic theorems}, the strong convergence of the "modulated" averages
$\frac1n \sum_{k=1}^n a_kT^kx$, for every weakly almost periodic $T$ and $x \in X$. 
Some general results are obtained in Section 2. In Section 3 we study modulation of 
flight vectors and weakly stable vectors of contractions in Hilbert spaces. 
In Section 4 we treat as an example the modulation by the the modified von Mangoldt function 
$\Lambda'(n)=\log n 1_\mathbb P(n)$, where $\mathbb P$ is the set of prime numbers.
In Section 5 we prove convergence of averages along the primes of doubly power-bounded operators on
$L^r$, $1<r<\infty$. In Section 6 we list some problems.
\medskip

\section{Modulated ergodic theorems for weakly almost periodic operators}

{\bf Definition.} A complex sequence $\ba:=(a_k)_{k>0}$ is called {\it Hartman almost periodic} 
(Hartman for short) if for every $|\lambda|=1$ the limit 
$c(\lambda):= \lim_n \frac 1n \sum_{k=1}^n  a_k\bar\lambda^k$ exists; $c(\lambda)$ is called
the Fourier (or Fourier-Bohr) coefficient function.
By Kahane \cite{Ka}, if $\ba$ is Hartman, then $\{\lambda: c(\lambda) \ne 0\}$ is countable.
\medskip

{\bf Definitions.} Let $T$ be weakly almost periodic on $X$. A vector $x \in X$ is called a 
{\it flight vector}  (or {\it weakly almost stable} in \cite{EFHN}) if $0$ is in the weak 
closure of its orbit; by the Krein-Shmulian theorem, this is equivalent to 
$T^{n_j}x\overset{w}\to 0$ for some subsequence $(n_j)$. The vector $x$ is called {\it (weakly)
stable} (for $T$) if $T^n x \to 0$ (weakly). The operator $T$ is called {\it (weakly) stable}
if every $x \in X$ is (weakly) stable for $T$.

Weakly almost periodic operators on a {\it complex} Banach space are characterized by the 
Jacobs-deLeeuw-Glicksberg decomposition \cite[pp. 105-106]{Kr}, \cite[Section 16.3 ]{EFHN}
\begin{equation} \label{jdg}
X=clm \{ y: Ty=\lambda y \text{ for some } \lambda \in \mathbb T\} \oplus 
\{z:  T^{n_j}z \overset{w} \to 0 \text{ for some } \{n_j\} \}.
\end{equation}
Here $clm$ means the closed linear manifold spanned.

\begin{lem} \label{product}
Let $\phi(x): \mathbb R^+ \longrightarrow \mathbb R^+$ be a strictly increasing function with
$\lim_{x\to\infty} \phi(x) = \infty$ and let $\ba$ be a complex sequence satisfying
\begin{equation} \label{Wphi}
s := \sup_n \frac1n \sum_{k=1}^n |a_k| \phi(|a_k|) < \infty.
\end{equation}
Then for every bounded complex sequence $(b_k)_{k>0}$ with $\frac1n\sum_{k=1}^n |b_k| \to 0$
we have 
\begin{equation} \label{limit}
\frac1n \sum_{k=1}^n |a_kb_k| \to 0 \quad \text{as } n\to \infty.
\end{equation}
\end{lem}
\begin{proof} When $\ba$ is bounded the lemma is trivial.
Without loss of generality $|b_k| \le 1$ for every $k$. 
For $\epsilon >0$ take  $C >0$ with $\frac{s}{\phi(C)} < \epsilon$, and let $N$ be such that 
for $n >N$ we have $\frac1n\sum_{k=1}^n|b_k| < \frac{\epsilon}{C}$. Then for $n>N$ we have
\begin{equation} \label{estimate}
\frac1n \sum_{k=1}^n |a_kb_k| =
\frac1n \underset{|a_k|  \le C}{\sum_{k=1}^n}|a_kb_k| +
\frac1n \underset{|a_k| >C}{\sum_{k=1}^n}|a_k|\phi(|a_k|) \cdot \frac{|b_k|}{\phi(|a_k|)} 
< 2\epsilon.
\end{equation}
\end{proof}

\noindent
{\bf Remarks.} 1. If $\ba$ satisfies (\ref{Wphi}), then $\sup_n \frac1n\sum_{k=1}^n |a_k| < \infty$,
since 
\begin{equation} \label{w-1}
\frac1n \sum_{k=1}^n |a_k| \le 
 \frac1n \underset{|a_k|  \le 1}{\sum_{k=1}^n}|a_k| + 
\frac1n \underset{|a_k| >1}{\sum_{k=1}^n}|a_k| \cdot \frac{\phi(|a_k|)}{\phi(1)} 
\le 1+\frac{s}{\phi(1)}.
\end{equation}
\smallskip

2. If $\ba$ satisfies (\ref{Wphi}), then $|a_n|=o(n)$. If not, then for some $c>0$ there exists 
$(n_j)$ increasing with $|a_{n_j}| \ge cn_j$. Then 
$$
\frac1{n_j} \sum_{k=1}^{n_j} |a_k|\cdot\phi(|a_k|) \ge \frac1{n_j}|a_{n_j}| \cdot \phi(|a_{n_j}|) \ge
c \cdot \phi(cn_j) \to \infty,
$$ 
contradicting (\ref{Wphi}).

\begin{theo} \label{flight}
Let $T$ be a weakly almost periodic operator on $X$ and let $\ba$  (real if $X$ is over $\mathbb R$)
satisfy (\ref{Wphi}).  Then for every flight vector $x$ we have 
$ \lim_n \|\frac1n\sum_{k=1}^n a_kT^k x\| =0$.
\end{theo}
\begin{proof} 
By Jones and Lin \cite{JL1}, $x$ is a flight vector  for $T$ if and only if 
\begin{equation} \label{uniform}
\sup_{\|x^*\| \le 1} \frac1n \sum_{k=1}^n |\langle x^*,T^kx \rangle| 
\underset{n \to\infty} \longrightarrow 0.
\end{equation}
We may assume $\|x\|=1$, and for $x^* \in X^*$ we put $b_k(x^*) = |\langle x^*,T^kx\rangle |$. Then
$$
\|\frac1n\sum_{k=1}^n a_k T^k x\| =
\sup_{\|x^*\|\le 1} |\frac1n\sum_{k=1}^n a_k \langle x^*, T^k x\rangle | 
\le \sup_{\|x^*\|\le 1} \frac1n\sum_{k=1}^n | a_k b_k(x^*)|.
$$
Since (\ref{uniform}) yields that for $n >N$ we have 
$\sup_{\|x^*\| \le 1} \frac1n\sum_{k=1}^n|b_k(x^*)| < \frac{\epsilon}{C}$, 
the proof of Lemma \ref{product} yields that
$ \sup_{\|x^*\|\le 1} \frac1n\sum_{k=1}^n | a_k b_k(x^*)| \to 0$, which proves the theorem.
\end{proof}

\noindent
{\bf Remark.} The special case of the theorem, with $\phi(x) =x^{p-1}$ for some $p>1$, was proved
differently (using H\"older's inequality) by \c C\"omez, Lin and Olsen \cite[Theorem 4.1]{CLO}.

\begin{cor} \label{complex}
Let $\ba$ be a Hartman sequence satisfying (\ref{Wphi}). Then for every weakly almost periodic
operator $T$ on a complex Banach space $X$ and every $x \in X$ 
the sequence $\frac1n\sum_{k=1}^n a_kT^kx$ converges in norm.
\end{cor}
\begin{proof}
We use the decomposition (\ref{jdg}). Theorem \ref{flight} applies to the space of flight vectors. 
For the eigenvectors we have the convergence since $\ba$ is Hartman, and on the closed space 
they generate convergence holds since $\sup_n \|\frac1n\sum_{k=1}^n a_kT^k\| \le
\sup_n \frac1n\sum_{k=1}^n|a_k| $, which is finite by the above remark.
\end{proof}

\noindent
{\bf Remark.} 
%
It was proved by Lin, Olsen and Tempelman \cite[Proposition 1.4]{LOT} that 
for every almost periodic $T$ on a complex Banach space $X$ the sequence $\frac1n\sum_{k=1}^n a_kT^kx$ 
converges in norm if and only if $\ba$ is a Hartman sequence with 
$\sup_n \frac1n\sum_{k=1}^n |a_k| < \infty$.
It was shown by Berend, Lin, Rosenblatt and Tempelman \cite[Example 2.7]{BLRT} that these necessary
conditions on $\ba$ are not sufficient for the convergence if $T$ is only {\it weakly} almost periodic.
\smallskip

\noindent
{\bf Definition.} Let $W_1$ be the set of sequences $\ba$ such that 
$\|\ba\|_{W_1}:= \limsup_n \frac1n\sum_{k=1}^n |a_k|$ is finite. Then $\|\cdot\|_{W_1}$ is a semi-norm.
We denote by $\mathcal A$ the set of sequences $\ba$ which satisfy (\ref{Wphi}) for some $\phi$ 
(which by (\ref{w-1}) are in $W_1$), and by $\overline{\mathcal A}^{W_1}$ its closure in the $W_1$-semi-norm.
By Theorem \ref{flight} and \cite[Theorem 1.2]{LOT}, if $\ba \in \overline{\mathcal A}^{W_1}$, 
i.e. there are sequences $\mathbf a^{(j)}$ which converge in the $W_1$-semi-norm to $\ba$ and 
each $\mathbf a^{(j)}$ satisfies (\ref{Wphi}) with some $\phi_j$, 
then $\ba$ modulates the flight vectors of every weakly almost periodic operator.

\begin{lem} \label{W1}
Let $(a_k)_{k\ge 1}$ be a complex sequence. Then the following are equivalent:

(i) $\sup_n \frac1n \sum_{k=1}^n |a_k| < \infty$.

(ii) For every sequence $b_n \to 0$ we have $\frac1n \sum_{k=1}^n |a_kb_k| \to 0$.

(iii) For every sequence $b_n \to 0$ the sequence $\big(\frac1n \sum_{k=1}^n a_kb_k )_n$ is bounded.

(iv) For every bounded sequence $(b_n)_n$ the sequence $\big(\frac1n\sum_{k=1}^n a_kb_k\big)_n$ 
is bounded.
\end{lem}
\begin{proof}
$(i) \Longrightarrow (ii)$: Let $b_n \to 0$. For $\epsilon >0$ take $N$ such that $|b_k| < \epsilon$
for $k>N$. Then for $n>N$ we have
$$
\frac1n \sum_{k=1}^n |a_kb_k| \le 
\frac1n\sum_{k=1}^N |a_kb_k|+\epsilon \frac1n\sum_{k=N+1}^n |a_k|,
$$ 
which yields (ii).

Trivially, $(ii) \Longrightarrow (iii)$ and $(i) \Longrightarrow (iv) \Longrightarrow (iii)$.

$(iii) \Longrightarrow (i)$: Assume (i) fails, so there exists   $(n_j)_j$ increasing with
$\frac1{n_j}\sum_{k=1}^{n_j} |a_k| >j$ for every $j$. Define $(b_n)$  for $N_j \le n < N_{j+1}$
by $|b_n| =1/\sqrt{j}$ and $\arg b_n = -\arg a_n$. Then $|b_n| \searrow 0$, and
$$
\big| \frac1{n_j} \sum_{k=1}^{n_j} a_kb_k\big|= \frac1{n_j} \sum_{k=1}^{n_j} |a_kb_k| \ge
\frac1{\sqrt{j}}\cdot \frac1{n_j}\sum_{k=1}^{n_j} |a_k| >\sqrt{j} \quad \forall j,
$$
contradicting (iii).
\end{proof}

\begin{prop} \label{stability}
Let $\ba=(a_k)_{k\ge 1}$ be a sequence of scalars. Then the following are equivalent:

(a) $\sup_n \frac1n \sum_{k=1}^n |a_k| < \infty$
\smallskip

(b) For every bounded linear $T$ on a (real or complex) Banach space $X$ and $x \in X$ weakly 
stable for $T$ we have  $\frac1n \sum_{k=1}^n a_kT^kx \to 0$ weakly.
\smallskip

(c) For every bounded linear $T$ on a (real or complex) Banach space $X$ and $x \in X$ 
stable for $T$ we have  $\frac1n \sum_{k=1}^n a_kT^kx \to 0$.
\end{prop}
\begin{proof}
Assume (a).  Let $T^nx \to 0$ weakly, and take a functional $x^* \in X^*$. 
Putting $b_n =\langle T^nx,x^*\rangle$ in (ii) of Lemma \ref{W1} we obtain (b).
Similarly, putting $b_n=\|T^nx\|$ when $T^n x \to 0$ we obtain (c).

Assume (b). Let $T$ be the shift $T(x_1,x_2\dots, ) =(x_2,x_3,\dots)$ on the complex $c_0$. Then 
for every $x \in c_0$ we have $T^nx \to 0$. By the assumption, $\frac1n \sum_{k=1}^n a_kT^kx \to 0$ 
weakly for every $x \in c_0$, so by the Banach-Steinhaus theorem (twice) 
$\sup_n \|\frac1n \sum_{k=1}^n a_kT^k \| < \infty$. Now the proof of \cite[Proposition 1.3]{LOT}
yields that $\sup_n \frac1n \sum_{k=1}^n |a_k| < \infty$.

The proof that (c) implies (a) is similar to (b) implies (a).
\end{proof}

\noindent
{\bf Remarks.} 1. In (b) we need not have strong convergence; see the example preceding Proposition
\ref{an=o(n)}.

2. When $T$ is power-bounded, the sets of its weakly stable vectors and of its stable vectors are closed
invariant subspaces.
\medskip

\begin{prop}
Let $\ba$ be a sequence such that for every weakly stable $T$ on a Banach space $X$ and 
$x \in X$ the sequence $\frac1n \sum_{k=1}^n a_kT^kx $ converges.
Then $\sup_n \frac1n \sum_{k=1}^n |a_k| < \infty$ and $\frac1{n^2} \sum_{k=1}^n |a_k|^2 \to 0$.
\end{prop}
\begin{proof}
The first condition follows from the proof of (b) implies (a) in Proposition \ref{stability}.
The second condition follows from Theorem \ref{eq-conditions} in the next section.
\end{proof}
\medskip

\begin{prop}
Let $\ba$ be a sequence such that for every weakly almost periodic $T$ on a Banach space $X$ and 
every flight vector $x \in X$ the sequence $\frac1n \sum_{k=1}^n a_kT^kx $ converges weakly.
Then $\sup_n \frac1n \sum_{k=1}^n |a_k| < \infty$ and $\frac1{n^2} \sum_{k=1}^n |a_k|^2 \to 0$.
\end{prop}
\begin{proof}
The first condition follows from the proof of (b) implies (a) in Proposition \ref{stability}.
The second condition follows from combining Proposition \ref{an=o(n)} and Lemma \ref{ak-squared},
proved in the next section.
\end{proof}
\medskip

\section{Modulation of flight vectors of contractions on a Hilbert space}

It was shown in \cite[Theorem 2.1]{LOT} that if $\ba$ is Hartman with
$\sup_n \frac1n\sum_{k=1}^n |a_k| < \infty$, then for every contraction $T$ on a complex Hilbert 
space $H$ and $x \in H$ the sequence $\frac1n\sum_{k=1}^n a_kT^kx$ converges in norm; (this is a 
consequence of Corollary 2.3 of \cite{BLRT}, which gives necessary and sufficient conditions). 
Note that (unlike the general result of \cite{LOT}) almost periodicity is not required.
It is therefore a natural question, when $\ba$ is not Hartman,  whether for contractions on $H$ 
the assertion of Theorem \ref{flight} holds when only $\sup_n \frac1n\sum_{k=1}^n |a_k| < \infty$; 
the example below yields a negative answer. 
\medskip

\begin{prop} \label{stable}
Let $\ba$ be a sequence of scalars satisfying $\sup_n \frac1n\sum_{k=1}^n |a_k| = C < \infty$.
If $|a_k| =o(k)$, 
then  for every contraction $T$ on a (real or complex) Hilbert space $H$ and $x \in H$ with 
$T^nx \overset{w}\to 0$ we have $ \lim _n \|\frac1n\sum_{k=1}^n a_k T^kx\| =0$.
\end{prop}
%
%
\begin{proof}
We first prove that $\frac1n\max_{1\le k\le n} |a_k| \to 0$ as $n \to \infty$. For
$\epsilon >0$ there is $K$ such that $\frac1k |a_k| < \epsilon$ for $k >K$. For $n >K$ we have
$$
\max_{1\le k \le n} \Big(\frac{|a_k|}n \Big)^2 \le \sum_{k=1}^n \frac{|a_k|^2}{n^2} \le
\frac1{n^2}\sum_{k=1}^K |a_k|^2 +\frac1n\sum_{k=K+1}^n |a_k|\frac{|a_k|}k \le
$$
\begin{equation} \label{equivalence}
\frac1{n^2}\sum_{k=1}^K |a_k|^2 + \epsilon \frac1n\sum_{k=N+1}^n |a_k| \le
\frac1{n^2}\sum_{k=1}^K |a_k|^2 + C\epsilon.
\end{equation}
Hence $\lim_{n\to\infty} \frac1n\max_{1\le k\le n} |a_k| =0$.

Define $\alpha_{n,0}=0, \quad \alpha_{n,k} =a_k/n$ for $1\le k \le n$, and $\alpha_{n,k}=0$
for $k>n$. Then $\lim_n \sup_k |\alpha_{n,k}| =0$ by the above, and Krengel's proof of the 
generalized Blum-Hanson theorem \cite[p. 254]{Kr} (where in (W5) one should read
$c:= \sup_N \sum_{i} |\alpha_{Ni}|< \infty$) yields the result.
\end{proof}

\noindent
{\bf Remark.} Without the additional condition $|a_k|=o(k)$, we have weak convergence in Proposition
\ref{stable}, by Proposition \ref{stability}; Theorem \ref{eq-conditions} below shows that in that 
case strong convergence does not hold.
\smallskip

\begin{cor} \label{stable+fix}
Let $\ba$ be a sequence of scalars satisfying $\sup_n \frac1n\sum_{k=1}^n |a_k| = C < \infty$.
If $c(1):=\lim_n \frac1n\sum_{k=1}^n a_k$ exists, then  for every contraction $T$
on a (real or complex) Hilbert space $H$  which converges in the weak operator topology,
necessarily to $E(T)$, we have $\lim_n \|\frac1n \sum_{k=1}^n a_kT^kx -c(1)E(T)x\|=0$ for 
every $x \in H$.
\end{cor}
\begin{proof} Existence of $c(1)$ yields \quad
$\displaystyle{ \frac{a_n}n = \frac1n\sum_{k=1}^na_k - \frac{n-1}n \cdot\frac1{n-1}\sum_{k=1}^{n-1}a_k
\to 0 }$.
\end{proof}
\smallskip

\begin{lem} \label{ak-squared}
Let $\ba$ satisfy $\sup_n \frac1n\sum_{k=1}^n |a_k| =C< \infty$. Then $a_n=o(n)$ if and only if
$\frac1{n^2}\sum_{k=1}^n |a_k|^2 \to 0$.
\end{lem}
\begin{proof}
One direction follows from $|a_n|^2/n^2 \le \frac1{n^2}\sum_{k=1}^n |a_k|^2 \to 0$.
The other direction is shown in (\ref{equivalence}).
\end{proof}

\begin{theo} \label{eq-conditions}
Let $H$ be a Hilbert space and let $\ba$ satisfy $\sup_n \frac1n \sum_{k=1}^n |a_k| < \infty$. 
Then the following are equivalent:
\smallskip

\noindent
(i) For every weakly stable contraction $T$ on $H$ and $x \in H$, \quad 
$\|\frac1n \sum_{k=1}^n a_kT^kx\| \to 0$.
\smallskip

\noindent
(ii) For some weakly stable unitary operator $U$ on $H$ and some $0 \ne x \in H$ we have
$\|\frac1n \sum_{k=1}^n a_kU^kx\| \to 0$.
\smallskip

\noindent
(iii) $a_n=o(n)$.
\smallskip

\noindent
(iv) $\frac1{n^2} \sum_{k=1}^n|a_k|^2 \to 0$.
\end{theo}
\begin{proof}
Clearly (i) implies (ii).

Assume (ii). We may assume $\|x\|=1$. The convergence $\|\frac1n \sum_{k=1}^n a_kU^kx\| \to 0$ 
implies that 
$$
\frac{|a_n|}n =\frac{\|a_n U^nx\|}n \to 0.
$$

Since we have $\sup_n \frac1n \sum_{k=1}^n |a_k| < \infty$, (iii) implies (i) by  
Proposition \ref{stable}.

(iv) is equivalent to (iii)  by Lemma \ref{ak-squared}.
\end{proof}

\noindent
{\bf Remarks.} 1. The condition $\sup_n \frac1n \sum_{k=1}^n |a_k| < \infty$ is used only for proving
(iii) implies (i). Although it is assumed also in Lemma \ref{ak-squared}, we show that (i) implies
(iv) without it:
Let $U$ be the shift on $\ell^2(\mathbb Z)$, defined by 
$U(\sum_{j=-\infty}^\infty b_j \vec e_j)=\sum_{j=-\infty}^\infty b_{j}\vec e_{j+1}$, where 
$(\vec e_j)_{j\in \mathbb Z}$ is the standard orthonormal basis.
Then $U$ is a weakly stable unitary operator, and by orthogonality of the orbit $(U^k\vec e_1)$,
(i) yields
$$
\frac1{n^2}\sum_{k=1}^n |a_k|^2 =\big\|\frac1n\sum_{k=1}^n a_k U^k\vec e_1\big\|^2 \to 0,
$$
so (iv) holds.

2.  Without the condition $\sup_n \frac1n \sum_{k=1}^n |a_k| < \infty$, (iii) does not imply (iv),
as shown by the simple example $a_k=\sqrt{k}$; since (i) implies (iv), this example shows that
without this boundedness condition, (iii) does not imply (i), (and then Proposition 
\ref{stable} fails).

3.  The condition $\sup_n \frac1n \sum_{k=1}^n |a_k| < \infty$ is not necessary for (i); see
\cite[Example 2.5]{BLRT}.

4. The proof of (ii) implies (iii) shows that if we have modulation of one flight vector of an
isometry in a Banach space, then $a_n=o(n)$. However, we'll show below that this condition, 
together with the condition $\sup_n \frac1n \sum_{k=1}^n |a_k| < \infty$, do not imply modulation 
of all flight vectors of unitary operators in a (complex) Hilbert space.

5. Condition (iv) is independent of the boundedness of the averages $\frac1n\sum_{k=1}^n|a_k|$.
The simple example $a_k=k^{1/4}$ satisfies (iv) but the averages are unbounded. In the next example
the averages are bounded but (iv) fails.
\smallskip

\begin{ex} {\it A sequence $\ba$ with $\sup_n \frac1n \sum_{k=1}^n |a_k| <\infty$ 
which does not modulate any weakly stable unitary operator on $H$.}

We define a non-negative sequence $\ba$ by $a_{2^j}=2^{j-1}$ for $j \ge 1$ and $a_k=0$ for all 
other indices $k \ge 1$. Then for $2^\ell \le n <2^{\ell+1}$ we have
$$\frac1n\sum_{k=1}^n a_k \le \frac1{2^\ell} \sum_{j=1}^\ell 2^{j-1} = \frac{2^\ell -1}{2^\ell} <1. $$
However, $a_{2^j}/2^j =\frac12$, so (iii) of the above theorem fails, hence also (ii) fails, so
$\ba$ does not modulate any weakly stable unitary operator.

On the other hand, by Proposition \ref{stability}, $\frac1n\sum_{k=1}^n a_kT^k x \to 0$ {\it weakly}
for every weakly stable operator $T$ on $H$.

\end{ex}
\smallskip

\begin{prop} \label{an=o(n)}
If for every unitary operator $T$ on a complex Hilbert space $H$ and every flight
vector $x$  we have that $\frac1n\sum_{k=1}^n a_kT^kx$ converges weakly, then $a_n=o(n)$.
\end{prop}
\begin{proof}
The assumption implies that for every unitary $T$ and every flight vector $x \in H$, the sequence
$\big( \frac1n \sum_{k=1}^n a_k \langle T^kx,x\rangle \big)_n$ converges, so 
\begin{equation} \label{to-zero}
\frac1n a_n\langle T^nx,x\rangle \to 0.
\end{equation}

Assume that for some $c>0$ and an increasing sequence $(n_j)_{j\ge 1}$ we have 
$|a_{n_j}| \ge c\cdot n_j$, and by taking a subsequence we may assume that $n_{j+1}/n_j$
tends to $\infty$. Such a sequence is a {\it rigidity sequence} \cite[Example 3.4]{EG}, i.e.
there exists a continuous probability $\mu$ on the unit circle $\mathbb T$ such that
\begin{equation}\label{eq:rigid}
\lim_{j\to\infty}\hat\mu(n_j) = 1.
\end{equation}
For properties and examples of rigidity sequences see Eisner  and Grivaux \cite{EG} and Bergelson, 
Del Junco, Lemanczyk and Rosenblatt \cite{BdJLR}. Note that every rigidity sequence has density zero
\cite[Proposition 2.12]{BdJLR}.

Let $H:=L^2(\mathbb T,\mu)$ and define on $H$ the multiplication operator $T$ by $Tf(z)=zf(z)$,
which is clearly unitary and has no (unimodular) eigenvalues. Hence the function $\mathbf 1$ is a 
flight vector for $T$, and $\langle T^{n_j}\mathbf{1},\mathbf{1}\rangle = \hat\mu(n_j) \to 1$. Hence 
$$
\frac{|a_{n_j}|}{n_j} |\langle T^{n_j}\mathbf{1},\mathbf{1}\rangle | \ge
c |\langle T^{n_j}\mathbf{1},\mathbf{1}\rangle | \to c >0,
$$
contradicting (\ref{to-zero}).
\end{proof}
\medskip

\noindent
{\bf Remarks.} 1. The previous example shows that the proposition is false if we assume the 
weak convergence only for weakly stable vectors, and not for all flight vectors.

2. Comparing Proposition \ref{an=o(n)} with (ii) implies (iii) in Theorem 
\ref{eq-conditions}, we require in the theorem modulation of {\it one} weakly stable vector of
a unitary operator, while the proposition requires only {\it weak} modulation of flight vectors,
but of {\it all} unitary operators. On the other hand, the theorem requires 
$\sup_n \frac1n\sum_{k=1}^n |a_k| < \infty$, which is not necssary for (i), and not assumed in the
proposition. A sequence without this condition, which satisfies the assumption of the 
proposition, is given in \cite[Example 2.5]{BLRT}.
\smallskip

\begin{ex}
{\it Sequences $\ba$ with $\sup \aveN |a_k|<\infty$ which do not modulate flight vectors
of a unitary operator.}
\end{ex}

We produce a way of constructing sequences $(a_n)\subset \R$ with 
\begin{equation}\label{eq:a}
\sup_N \aveN |a_k|<\infty
\end{equation}
which fail to be good modulating sequences for unitary operators without eigenvalues. 

Let $(k_n)\subset\N$ be a rigidity sequence. 
%

Define $k_0:=0$ and
$$
|a_n|:=\begin{cases}
k_l-k_{l-1},\quad &n=k_l\ \text{ for }l\geq 1,\\
0 \quad &\text{otherwise.}
\end{cases}
$$
Then for $k_n \le N < k_{n+1}$ we have
\begin{equation}\label{eq:ave=1}
\frac1N\sum_{l=1}^N |a_l| \le
\frac1{k_n}\sum_{l=1}^{k_n}|a_l|
=\frac1{k_n}\sum_{l=1}^{n}|a_{k_l}|
=\frac1{k_n}\sum_{l=1}^{n} (k_l-k_{l-1})=1
\end{equation}
and (\ref{eq:a}) is satisfied. 

Consider now the complex Hilbert space $L^2(\T,\mu)$ with the multiplicaton operator $T$ defined 
on it by $(Tf)(z):=zf(z)$. Since $\hat{\mu}(n)=\langle T^n \mathbf{1},\mathbf{1}\rangle$, $n\in\N$, 
it is enough to make the averages 
$$
\aveN a_k \hat{\mu}(k)
$$
diverge. We will define a rule for each $s_n:=\text{sign}(a_{k_n})$ to be $\pm 1$ later. Observe
\begin{equation}\label{eq:sum}
\frac1{k_n}\sum_{l=1}^{k_n} a_l \hat{\mu}(l)=
\frac1{k_n}\sum_{l=1}^n a_{k_l} \hat{\mu}(k_l).
\end{equation}
By (\ref{eq:rigid}) and (\ref{eq:ave=1}), Lemma \ref{W1}(ii) yields that the convergence of 
the right hand side of (\ref{eq:sum}) is equivalent to the convergence of 
$$
\frac1{k_n}\sum_{l=1}^{n}a_{k_l}=
\frac1{k_n}\sum_{l=1}^n (k_l-k_{l-1}) s_l
$$
which is the $k_n$th Ces\`aro average of the sequence 
$$
s_1,\ldots,s_1,s_2,\ldots,s_2,s_3,\ldots,s_3,\ldots,
$$
where $s_l$ appears exactly $k_l-k_{l-1}$ times. 
Now it is easy to define $s_l$ to be $1$ or $-1$ such that these averages diverge (first make 
the averages close to $1$, then close to $-1$ etc.). 

Note that every such sequence $(a_n)$ is by construction Ces\`aro divergent and therefore not Hartman.

As a concrete example one can take $k_n:=2^n$, which is a rigidity sequence by Eisner  and Grivaux 
\cite[Prop.~3.9]{EG}, or Bergelson, Del Junco, Lemanczyk  and Rosenblatt \cite[Prop.~3.27]{BdJLR}.

\begin{ex} {\it  A sequence $\ba$ with $\sup \aveN |a_k|<\infty$  and $a_k=o(k)$ which does
 not modulate flight vectors of a unitary operator.}

In the previous example, take a rigidity sequence $(k_n)$ with 
$\lim_{n\to \infty}\frac{k_{n+1}}{k_n}=1$, which exists by \cite[Example 3.17]{EG}. Then,
in addition to (\ref{eq:a}), $(a_n)$ satisfies $a_n=o(n)$, since
$$
\frac{|a_{k_n}|}{k_n}=\frac{k_n-k_{n-1}}{k_n}=1-\frac{k_{n-1}}{k_n}\to 0\quad \text{as }n\to \infty.
$$

This example shows that good modulating sequences for weakly stable operators (as in Proposition 
\ref{stable}) need  not modulate flight vectors of all unitary operators. In particular, the 
sequence in this example does not satisfy (\ref{Wphi}) for any $\phi$ as in Lemma \ref{product}.
\end{ex}
\medskip

\begin{remark}
Note that $(k_n)$ is a rigidity sequence if and only if there exists a weakly mixing system 
$(\Omega,\nu,T)$ which is rigid along $(k_n)$, i.e., $\lim_{n\to\infty}T^{k_n}=I$ in the strong 
operator topology. (In fact, this is one of the several equivalent definitions of rigidity 
sequences). So the unitary operator can be constructed to be the Koopman operator of a weakly 
mixing transformation restricted to the orthogonal complement of $\textbf{1}$ and the averages
$$
\aveN a_k T^kf
$$
diverge for every non-zero $f$. 
\end{remark}

\medskip

\section{Example: Modulation by the von Mangoldt function}

Weighted ergodic theorems for dynamical systems, where the weights are given by some arithmetic 
function, were obtained by Cuny and Weber \cite{CW}. In weighted ergodic theorems we have a 
non-negative sequence of weights $(w_n)$ with {\it diverging} partial sums $W_n:=\sum_{k=1}^n w_k$, 
and one considers convergence of the weighted averages $\frac1{W_n} \sum_{k=1}^n w_k T^kf$. 
When $W_n/n$ converges to a (finite) non-zero limit, convergence of weighted averages and of
averages modulated by $(w_n)$ coincide.

\medskip

In this section we consider modulation by the von Mangoldt function $\Lambda$, defined by
$$
\Lambda(n):=\begin{cases}
\log p, & n=p^k \text{ for some } p \in\mathbb P,\ k\in \mathbb N,\\
0 & \text{otherwise}
\end{cases}
$$
and by its simplified version $\Lambda'$ defined by $ \Lambda'(n):= 1_\P(n)\log n $
($\mathbb P$ denotes the set of primes).
The prime numbers theorem implies $\limaveN \Lambda(k) = 1$  \cite[pp. 56,118]{Da} 
\cite[p. 79]{Ap}, and also
(see \cite[p. 79]{Ap})
\begin{equation}\label{eq:Lambda'-mean}
\limaveN \Lambda'(k) = 1.
\end{equation}
Hence (see also \cite[formula (5)]{FHK}), we have
\begin{equation}\label{eq:Lambda-diff}
\lim_{N\to\infty} \left|\aveN \Lambda(k) - \aveN \Lambda'(k)   \right|=
\lim_{N\to\infty} \aveN (\Lambda(k) - \Lambda'(k))= 0.
\end{equation}


The proof of the next lemma is included for the sake of completeness
(making precise the proof of \cite[Lemma 1]{FHK}).

\begin{lem} \label{fkh-lemma}
For every $C>0$ we have
\begin{equation}\label{eq:Lambda-primes}
\lim_{N\to\infty} \sup_{\|(b_k)\|_\infty \le C}
\left| \aveN \Lambda'(k)b_k -\frac1{\pi(N)} \sum_{p\leq N,\ p\in\P} b_p\right|=0,
\end{equation}
where $\pi(N)$ is the number of primes not exceeding $N$. 
\end{lem}
\begin{proof} 
Let $p_1, p_2, \dots$ be the sequence of primes in increasing order.
By the prime numbers theorem, $\lim_N \frac{\pi(N)\log N}N=1$, so for $\epsilon >0$ 
there exists $N_0$ such that $\frac{\pi(N) \log N}N < 1+\epsilon$ for $N>N_0$. 
For $N>N_1 \ge N_0$ also $|1-\frac1N\sum_{j=1}^{\pi(N)} \log p_j | < \epsilon$,
by (\ref{eq:Lambda'-mean}).
For $(b_k)$ bounded with $\|(b_k)\|_\infty \le C$ and $N>N_1$ we therefore have
$$
\left|\frac1{\pi(N)} \sum_{p\le N, p\in \mathbb P} b_p - \frac1N\sum_{k=1}^N \Lambda'(k)b_k \right|=
\left| \sum_{j=1}^{\pi(N)} \Big[ \frac1{\pi(N)} - \frac{\log p_j}N \Big]b_{p_j}\right| \le
$$
$$
\left| \sum_{j=1}^{\pi(N)} \Big[ \frac{1+\epsilon}{\pi(N)} - \frac{\log p_j}N \Big]b_{p_j}\right| +
\epsilon\left| \frac1{\pi(N)} \sum_{j=1}^{\pi(N)}b_ {p_j}\right| \le
$$
$$
 \sum_{j=1}^{\pi(N)} \Big[ \frac{1+\epsilon}{\pi(N)} - \frac{\log p_j}N \Big] \|(b_k)\|_\infty +
\epsilon\|(b_k)\|_\infty =
$$
$$
 \left(1-\frac1N \sum_{j=1}^{\pi(N)} \log p_j\right)\|(b_k)\|_\infty + 2\epsilon\|(b_k)\|_\infty
\le 3\epsilon \|(b_k)\|_\infty  \le 3\epsilon C.
$$
This shows that
\begin{equation} \label{uniform1}
\sup_{\|(b_k)\|_\infty \le C} 
\left|\frac1{\pi(N)} \sum_{p\le N, p\in \mathbb P} b_p - \frac1N\sum_{k=1}^N \Lambda'(k)b_k \right| 
\ \underset{N\to\infty}\longrightarrow 0.
\end{equation}
\end{proof}

{\bf Remark.}
By (\ref{eq:Lambda-diff}) and (\ref{eq:Lambda-primes}), for every $C>0$ we have also
$$
\lim_{N\to\infty} \sup_{\|(b_k)\|_\infty \le C} 
\left|\aveN \Lambda(k)b_k -\frac1{\pi(N)} \sum_{p\leq N,\ p\in\P} b_p\right|=0.
$$

\begin{prop} \label{limits}
The following are equivalent for a power-bounded operator $T$ on a Banach space $X$ and $x \in X$:
\smallskip

(i) $\frac1n\sum_{k=1}^n T^{p_k}x$ converges.
\smallskip

(ii) $\frac1N \sum_{j=1}^{\pi(N)} \log p_j T^{p_j}x=\frac1N \sum_{k=1}^N \Lambda'(k)T^kx$ converges.
\smallskip

(iii) $\frac1N \sum_{k=1}^N \Lambda(k)T^kx$ converges.
\smallskip

\noindent
If either limit exists, then the three limits are the same.
\end{prop}
\begin{proof} 
To prove the equivalence of (i) and (ii), let $M := \sup_n \|T^n\|$, and fix $x \in X$. 
For $x^* \in X^*$ put $b_k(x^*)= \langle x^*,T^kx\rangle$. Then (\ref{uniform1}) yields
$$
\left\| \frac1{\pi(N)}\sum_{k=1}^{\pi(N)} T^{p_k}x-\frac1N \sum_{k=1}^N \Lambda'(k)T^kx \right\| =
\sup_{\|x^*\| \le 1} 
\left| \frac1{\pi(N)}\sum_{k=1}^{\pi(N)} b_{p_k}(x^*) -
\frac1N \sum_{k=1}^N \Lambda'(k)b_k(x^*) \right|  \le
$$
$$
\sup_{\|(b_k)\|_\infty \le M\|x\|} 
\left|\frac1{\pi(N)} \sum_{p\le N, p\in \mathbb P} b_p - \frac1N\sum_{k=1}^N \Lambda'(k)b_k \right| 
\to 0.
$$

For the equivalence of (ii) and (iii), we use (\ref{eq:Lambda-diff}), noting that 
$\Lambda(k) \ge \Lambda'(k)$, and obtain
$$
\left\| \frac1N \sum_{k=1}^N \Lambda(k)T^kx - \frac1N \sum_{k=1}^N \Lambda'(k)T^kx \right\| \le
\frac1N \sum_{k=1}^N \left\|\big(\Lambda(k) -  \Lambda'(k)\big) T^k x \right\|\le
$$
$$
\left(\frac1N \sum_{k=1}^N \Lambda(k) - \frac1N \sum_{k=1}^N \Lambda'(k) \right) M\|x\|\to 0.
$$

Our proof shows that if either limit exists, then the three limits are the same.
\end{proof}

\noindent
{\bf Remarks.} 1. The above proof shows also the equivalence of the above three conditions when 
the convergence is taken {\it weakly.}

2. The equivalence of (i) with (ii), or (iii), transforms the averages 
along the primes into modulated (or, by \eqref{eq:Lambda'-mean}, weighted) averages.
\medskip

For $\lambda \in \mathbb T$, (\ref{eq:Lambda-diff}) yields
$\lim_N\big|\frac1N\sum_{k=1}^N \Lambda(k)\lambda^k - \frac1N\sum_{k=1}^N \Lambda'(k)\lambda^k\big|=0$,
so $(\Lambda(n))$ is Hartman if and only if $(\Lambda'(n))$ is, and in that case they have 
the same Fourier-Bohr coefficients.

\begin{theo} \label{prime-average}
(i) The sequences $(\Lambda(n))$ and $(\Lambda'(n))$ are Hartman, and have the same Fourier-Bohr
coefficients $c(\lambda)$.

(ii) For every contraction $T$ on a complex Hilbert space $H$ and $x \in H$, we have the 
convergence
$$
\lim_{n \to \infty} \frac1n \sum_{j=1}^n T^{p_j}x =
\lim_{N \to\infty} \frac1N \sum_{j=1}^{\pi(N)} \log p_j T^{p_j}x =
\limaveN \Lambda(k) T^k x =
$$
\begin{equation} \label{identify}
\sum_{\lambda \in \sigma(\Lambda(n))} c(\lambda)E(\bar\lambda)x, 
\end{equation}
where $E(\lambda)$ is the orthogonal projection on the eigenspace corresponding to 
$\lambda \in \mathbb T$.

In particular, for every flight vector $x\in H$,
$$
\lim_{n \to \infty} \frac1n \sum_{j=1}^n T^{p_j}x =
\lim_{N \to\infty} \frac1N \sum_{j=1}^{\pi(N)} \log p_j T^{p_j}x =
\limaveN \Lambda(k) T^k x =0.
$$
\end{theo}
\begin{proof}
We first prove the assertion (i) for $(\Lambda'(n))$.
By (\ref{eq:Lambda'-mean}), $\limaveN \Lambda'(k) = 1$, so by positivity,
$\sup_N \frac1N\sum_{k=1}^N|\Lambda'(k)| <\infty$.

It is easily checked that 
$\frac1{\pi(N)} \sum_{p\leq N,\ p\in\P} \lambda^p = \frac1{\pi(N)}\sum_{k=1}^{\pi(N)} \lambda^{p_k}$ 
for $\lambda \in \mathbb T$. 
The uniform distribution mod 1 of $(p_k\theta)$ for every $\theta$ irrational, proved by Vinogradov, 
is equivalent (by Weyl's criterion) to the convergence to zero of $\frac1n\sum_{k=1}^n \lambda^{p_k}$
for every $\lambda \in \T$ not a root of unity, which yields
$\frac1{\pi(N)} \sum_{p\leq N,\ p\in\P} \lambda^p \to 0$ for such $\lambda $; see also 
\cite[Lemma 4.1]{CEF}. For $\lambda \in \T$ a root of unity, dividing  by $\pi(N)$ formula (3.5) 
in \cite[p. 180]{Pr}, we obtain that $\frac1{\pi(N)} \sum_{p\leq N,\ p\in\P} \lambda^p$ 
converges (with an identification of the limit).

Hence, by (\ref{eq:Lambda-primes}), for every $\lambda\in\T$ there exists the limit
\begin{equation} \label{coeff}
 c(\lambda) := \limaveN \Lambda'(k)\lambda^k 
=\lim_{N\to\infty}\frac1{\pi(N)} \sum_{p\leq N,\ p\in\P} \lambda^p.
\end{equation}
This shows that $(\Lambda'(n))$ is Hartman. 
(Note that $c(\lambda)=0$ for every $\lambda$ which is not a root of unity; for an explicit formula 
when $\lambda$ is a root of unity see, e.g., \cite[p. 180]{Pr} or \cite[equation (7)]{CEF}).  
Since $\Lambda(n) \ge \Lambda'(n)$, we conclude by (\ref{eq:Lambda-diff}) that also $(\Lambda(n))$
is Hartman, with the same coefficients $c(\lambda)$.

(ii) The convergence of $\frac1n\sum_{k=1}^n \lambda^{p_k}$ for every 
$\lambda \in \mathbb T$ obtained above yields, by \cite[Theorem 4.4]{LOT}, that for every
contraction $T$ on a Hilbert space $\frac1n\sum_{k=1}^n T^{p_k}x$ converges strongly for every 
$x \in H$, and since in (\ref{coeff}) $c(\lambda) \ne 0$ only for countably many $\lambda$ 
(roots of unity),
$$
\lim_{N \to\infty} \frac1{\pi(N)} \sum_{j=1}^{\pi(N)}  T^{p_j}x = \limaveN T^{p_k}x =
\sum_{\lambda \in \mathbb T \text{ root of 1}} c(\lambda)E(\bar\lambda)x, \quad \forall x\in H.
$$
Existence of the other limits in (ii) and their equality follow from Proposition \ref{limits}.
%
\end{proof}

\noindent
{\bf Remarks.} 
%
%
1. If $T$ is invertible on $H$ with $T$ and $T^{-1}$ both power-bounded, then the assertion of (ii) 
holds, since by Sz.-Nagy \cite{S-N} $T$ is similar to a unitary operator.

2. If $T$ is almost periodic on a complex Banach space, the averages along the primes 
$\frac1{\pi(N)} \sum_{j=1}^{\pi(N)}  T^{p_j}x $ converge strongly, because for $T$ almost periodic
$x \in X$ is a flight vector if and only if $\|T^n x\| \to 0$, 
and convergence for the eigenvectors  follows from (\ref{coeff}).
Proposition \ref{limits} yields convergence of the modulated averages  
$\frac1N \sum_{k=1}^N \Lambda(k)T^k x$ and $\frac1N \sum_{k=1}^N \Lambda'(k)T^k x$.

\begin{theo} \label{realH}
Let $T$ be a contraction on a real Hilbert space $H$. 
Then for every $x \in H$ we have norm convergence of 
$\frac1n\sum_{j=1}^n T^{p_j}x$ and of $\frac1N\sum_{j=1}^{\pi(N)}\log p_j T^{p_j}x$.
\end{theo}
\begin{proof} 
Given a real Hilbert space $H$, we define its complexification $H_C :=H \times H$ with 
the usual addition of pairs, and multiplication by complex scalars defined by
$$(\alpha +i\beta)(x,y) := (\alpha x - \beta y, \alpha y + \beta x).$$
A scalar product on $H_C$ is defined by 
$$
\langle (x,y),(u,v)\rangle := 
\langle x,u\rangle + \langle y,v \rangle +i(\langle y,u \rangle - \langle x,v \rangle)\ ,
$$
so the norm on $H_C$ is given by $\|(x,y)\|^2 =\|x\|^2 + \|y\|^2$ (see \cite[p. 98]{MDW} for proofs).
Given a linear operator $T$ on (the real) $H$ we define its extension to $H_C$ by 
$T_C(x,y):=(Tx,Ty)$. Clearly $\|T_C\| \le \|T\|$, and since $\|T_C(x,0)\|=\|Tx\|$, we have 
$\|T_C\|=\|T\|$.

Given a contraction on the real Hilbert space $H$, we apply Theorem \ref{prime-average}(ii) to 
$T_C$ on $H_C$ and obtain the assertion of our theorem.
\end{proof}

{\bf Remarks.} 1. The deep result of Bourgain \cite{B2},\cite{B} and Wierdl \cite{W} on the 
pointwise ergodic theorem along primes for $f \in L^r$ ($r>1$) of a probability preserving system 
(together with Wierdl's \cite[Lemma 1]{W}) shows that for any probability 
preserving system $(\Omega,\Sigma,\mu,\tau)$ and $f \in L^r(\Omega,\mu)$, $r>1$, we have that
$ \frac1{N} \sum_{j=1}^{\pi(N)} \log p_jf(\tau^{p_j}\omega)$ converges a.e. (see formula (2) in 
\cite{W}), i.e. $(\Lambda'(n))$ is in fact a good weight sequence for the pointwise ergodic theorem 
in $L^r$. For $r=2$ the limit is identified by the series in (\ref{identify}) (where 
$Tf =f\circ\tau$).  It can be shown that also $(\Lambda(n))$ is a good weight sequence for the 
pointwise ergodic theorem in $L^r$.

2. For pointwise ergodic theorems with other arithmetic weights see El Abdalaoui, Kulaga-Przymus, 
Lema\'nczyk, de la Rue \cite[Section 3]{EKLR} and Cuny, Weber \cite{CW}.

3. In his thesis \cite{W2}, Wierdl extended the result of \cite{W} to a.e. convergence of averages 
of $L^2$ functions along fixed powers of primes. This result was rediscovered by Nair \cite{N1} 
(also for $L^2$ functions), and extended in \cite{N2} to averages  of $L^p$ functions ($p>1$) 
along $Q(p_j)$, where $Q(t)$ is a polynomial with non-negative integer coefficients.

\begin{theo}
Let $Q(t)$ be a polynomial with non-negative integer coefficients.
Then for every contraction $T$ on a (real or complex) Hilbert space $H$ and $x \in H$, 
the averages $\frac1n\sum_{j=1}^n T^{Q(p_j)} x$ converge in norm.
\end{theo}
\begin{proof} 
We first prove the theorem for a complex Hilbert space $H$. For 
$\lambda = e^{2\pi i\alpha} \in\mathbb T$ which is not a root of unity (i.e $\alpha$ irrational), 
the sequence $(\alpha Q(p_j))$ is equidistributed  modulo 1, by Vinogradov (see Rhin \cite{Rh}), 
so by Weyl's criterion $\frac1n\sum_{j=1}^n \lambda^{Q(p_j)} \to 0$.

For a root of unity $\lambda = e^{2\pi i\frac{b}q}$ with $(b,q)=1$, note that if $p\equiv a$
modulo $q$, then $\lambda^{Q(p)}=\lambda^{Q(a)}$. 
We now use the prime number theorem in arithmetic progressions \cite[p. 598]{Wa} (see also 
\cite[p. 133]{Da}).  Let $\pi(x,q,a):= |\{p\in\mathbb P: p\le x,\ p\equiv a\ \mod q\}|$ and denote
$Li(x):= \int_2^x\frac1{\log t}dt$.  If $(a,q)=1$, then for  some $s<1$, for large $x$ we have 
\begin{equation} \label{walfisz}
\Big| \pi(x,q,a) - \frac{Li(x)}{\phi(q)}\Big| \le A\frac x{\e^{A\sqrt{x}}} +A\frac{x^s}{\log x}.
\end{equation}
Using the notation $e(x):=\e^{2\pi ix}$ we have 
$$
\frac1N \sum_{j=1}^N \lambda^{Q(p_j)} = \frac1N \sum_{j=1}^N e\big(Q(p_j)\frac{b}q\big) = 
\sum_{a=0}^{q-1} \frac1N \underset{p_j \equiv a \mod q } {\sum_{j\le N}} e\big(Q(a)\frac{b}q\big) =
$$
$$
\sum_{a=0}^{q-1} e\big(Q(a)\frac{b}q\big) \cdot \frac{\pi(p_N ,q,a)}N.
$$
Since $\pi(x,q,a)$ is at most 1 when $(a,q) >1$, we compute the limit as $N\to \infty$ for $(a,q)=1$;
since $\frac{Li(x)}{\pi(x)} \to 1$ as $x \to \infty$ (e.g. \cite[p. 102]{Ap}) and $\pi(p_N)=N$, 
(\ref{walfisz}) yields 
$$
\lim_{N\to \infty} \frac{\pi(p_N ,q,a)}N = 
\frac1{\phi(q)}\lim_{N\to\infty} \frac{Li(p_N)}{\pi(p_N)} = \frac1{\phi(q)},
$$  
so that
$$
\lim_{N\to \infty} \frac1N \sum_{j=1}^N \lambda^{Q(p_j)} =
\frac1{\phi(q)} \underset{(a,q)=1}{\sum_{0<a<q}} e\big(Q(a)\frac{b}q\big).
$$
We therefore have that for every $\lambda \in \mathbb T$ the sequence 
$\frac1n\sum_{j=1}^n \lambda^{Q(p_j)}$ converges. By \cite[Theorem 4.4]{LOT}, 
$\frac1n \sum_{j=1}^n T^{Q(p_j)}x$ converges for every contraction $T$ on $H$ and $x \in H$.
\medskip

When $H$ is a real Hilbert space, we use the complexification, as in the proof of Theorem \ref{realH},
to deduce the result from the complex case.
\end{proof}

{\bf Remark.} A quick proof for the complex case, based on later results in ergodic theory, is this:
For any $\lambda \in \mathbb T$, we apply Nair's result \cite[Theorem 1]{N2} to the unit circle 
rotation $\theta(z)=\lambda z$ and the function $f(z)=z$, and deduce that 
$\frac1n\sum_{j=1}^n \lambda^{Q(p_j)} $ converges. Theorem 4.4 of \cite{LOT} yields
that if $T$ is a contraction on $H$, the desired convergence holds.
\bigskip

By Theorems \ref{prime-average} and \ref{realH}, $(\Lambda'(n))_n$ is a Hartman sequence 
which modulates all contractions in a (real or complex) Hilbert space.
The following shows that $(\Lambda'(n))_n$ cannot be approximated in the $W_1$-semi-norm 
by sequences satisfying (\ref{Wphi}). Thus the set of sequences which modulate all 
flight vectors of Hilbert space contractions is strictly larger than 
$\overline{\mathcal{A}}^{W_1}$ defined in Section 2.
\begin{prop}
Every $(a_n)\in \mathcal{A}$ satisfies $\|a_n-\Lambda'(n)\|_{W_1}\geq \frac{1}{2}$. In particular, 
both $(\Lambda'(n))$ and $(\Lambda(n))$ do not belong to $\overline{\mathcal{A}}^{W_1}$.
\end{prop}
\begin{proof}
Let $(a_n)$ satisfy  
\begin{equation*}
d:=\|(a_n)-\Lambda'\|_{W_1}<\frac{1}{2}.
\end{equation*}
We have to show that $(a_n)\notin \mathcal{A}$. Define $(b_n)$ by 
$$
b_n:=|a_n| 1_M(n),
$$
where
$$
M:=\left\{p\in\P:\ |a_p|\geq \frac{\log p}2\right\}.
$$
Observe
$$
|b_n-\Lambda'(n)|=
\begin{cases}
||a_n|-\Lambda'(n)|\leq |a_n-\Lambda'(n)|, & n\in M,\\
\log n\leq 2 |a_n-\Lambda'(n)|, & n\in\P\setminus M,\\
0, & n\notin\P.
\end{cases}
$$
Therefore $(b_n)$ satisfies
\begin{equation}\label{eq:2d}
\|(b_n)-\Lambda'\|_{W_1}\leq 2d <1.
\end{equation}
Moreover, since $(b_n)$ is dominated by $(|a_n|)$, it suffices to show that $(b_n)\notin \mathcal{A}$.

Let $\phi:\R_+\to\R_+$ satisfy $\phi\nearrow \infty$ and fix $C>0$. 
Observe that $b_p\geq C$ for every $p\in M$ with $p\geq e^{2C}$ and therefore for $N>e^{2C}$ we have
$$
\aveN b_k \phi(b_k) \geq \frac1N \sum_{n\in M\cap[e^{2C}, N]} b_n \phi(b_n) \geq 
\frac{\phi(C)}N \sum_{n\in M\cap[e^{2C}, N]} b_n =
\frac{\phi(C)}N \sum_{n\in[e^{2C}, N]} b_n.
$$
This together with (\ref{eq:Lambda'-mean}) and (\ref{eq:2d}) implies
\begin{eqnarray*}
\liminf_{N\to\infty}\aveN b_n \phi(b_n) &\geq& \phi(C) \left(\limaveN \Lambda'(n) - \limsup_{N\to\infty}\aveN|\Lambda'(n)-b_n|\right)\\
&\geq& (1-2d)\phi(C).
\end{eqnarray*}
Since $C>0$ was arbitrary, we obtain $\limaveN b_n \phi(b_n)=\infty$ and hence $(b_n)\notin\mathcal{A}$.
\end{proof}
%
%
%

\section{Averages along the primes for operators on $L^r$, $1<r<\infty$.}

We now want to extend the convergence  of the averages along the primes, as in Theorem 
\ref{prime-average}(ii,) to positive contractions of $L^r$, $1<r<\infty$, 
and to invertible operators $T$ on $L^r$ with $T$ and $T^{-1}$ power-bounded.
\smallskip

{\bf Definition.} A linear operator on (the real or complex) $L^r(\Omega,\Sigma,\mu)$ of a 
$\sigma$-finite measure space is called a {\it Lamperti operator}  (or {\it disjointness preserving})
if it preserves disjointness of supports, i.e. $Tf$ and $Tg$ have disjoint supports if 
$f$ and $g$ have disjoint supports.
\smallskip

{\bf Definition.} $T$ on $L^r$ is called a  {\it quasi-isometry} if there exist constants
$c_1$ and $c_2$, and an increasing sequence of positive integers  $(M_n)_{n\ge 1}$ such that
for every $f \in L^r$ we have
$$
c_1\|f\|^r \le \frac1{M_n} \sum_{k=1}^{M_n} \|T^k f\|^r \le  c_2\|f\|^r, \qquad \forall n \ge 1.
$$

If $T$ is invertible with both $T$ and $T^{-1}$ power-bounded, then $\|f\| \le K\|T^k f\|$,
so $T$ is a quasi-isometry. If $T$ is similar to a quasi-isometry, it is a quasi-isometry.

\begin{prop} \label{kan}
Let $(\Omega,\Sigma,\mu)$ be a $\sigma$-finite measure space and fix $1<r<\infty$. Let $T$ be a 
Lamperti quasi-isometry on $L^r(\mu)$. Then there exists $C>0$ such that for every $f \in L^r(\mu)$
\begin{equation} \label{max}
\big\| \sup_n  |\frac1n\sum_{k=1}^n T^k f|\, \big\|_r \le C\|f\|_r, 
\end{equation}
and $\frac1n\sum_{k=1}^n T^kf$ converges a.e. and in norm.
\end{prop}
\begin{proof} The maximal inequality is obtained by transfering the maximal inequality for the shift
on $\mathbb Z$, using \cite[Theorem 2.1]{JOW}. The a.e. convergence follows from \cite{JOW} -- the
needed inequality (2.7) there was proved by Bourgain \cite{B2}. Together with (\ref{max}), the a.e. 
convergence yields the norm convergence.
\end{proof}

\begin{theo} \label{Lr}
Let $(\Omega,\Sigma,\mu)$ be a $\sigma$-finite measure space and fix $1<r<\infty$. Let $T$
be a linear operator on $L^r(\mu)$ satisfying one of the following conditions:

(i) $T$ is a positive contraction or a Lamperti contraction.

(ii) $T$ is an isometry.

(iii) $T$ is invertible, with both $T$ and $T^{-1}$ positive and power-bounded.

(iv)  $T$ is a Lamperti quasi-isometry.

Then for every $f \in L^r$ we have convergence in norm of $\frac1n \sum_{j=1}^n T^{p_j}f$ and of 
\newline
$\frac1N \sum_{j=1}^{\pi(N)} \log p_j T^{p_j}f$.
\end{theo}
\begin{proof}
We first show that it is enough to prove  the theorem when $T$ satisfies (iv). 

If $T$ satisfies (iii), then $T$ is a quasi-isometry as noted above, and the positivity yields 
that $T$ is Lamperti, by Kan \cite[Proposition 3.1]{Kan}, so (iv) is satisfied.

If $T$ satisfies (ii) and $r \ne 2$, then the isometry $T$ is a Lamperti operator 
\cite[Theorem 3.1]{La} ($T$ need not be positive).  When $r=2$ the theorem  follows from Theorem 
\ref{prime-average} or Theorem \ref{realH}.

When $T$ is a positive contraction, we use Akcoglu's dilation of $T$ to a positive isometry 
(see \cite{AS}), so the theorem in that case follows from its conclusion for isometries.
When $T$ is a Lamperti contraction, it can be dilated to a (Lamperti) isometry by Kan's dilation
\cite[Theorem 4.4]{Kan}, so the theorem follows from its conclusion for isometries.
\medskip

Jones, Olsen and Wierdl \cite{JOW} proved that if $T$ is a Lamperti quasi-isometry on $L^r$, then 
$\frac1n \sum_{j=1}^n T^{p_j}f$ converges a.e. for every $f \in L^r$. However, in their proof they 
refer to \cite{W} for the maximal inequality needed in their formula (2.8); but the maximal 
inequality proved in \cite{W} is for the modulated averages, modulated by $(\Lambda'(k))$. 
(Their formula (2.7) is proved in \cite{B2} for the averages along the primes). We are grateful 
to M\'at\'e Wierdl for showing us the equivalence of these maximal inequalities (see details in the 
Appendix).  With the correct maximal inequality,  when transferred to $L^r$ using 
\cite[Theorem 2.1]{JOW}, we have that $\sup_n |\frac1n\sum_{j=1}^n T^{p_j}f|$ is in $L^r$, so the 
a.e. convergence proved in \cite{JOW} implies norm convergence, by Lebesgue's bounded convergence 
theorem.  This proves the theorem for the averages along the primes. The other convergence follows 
from Proposition \ref{limits}.
\end{proof}

\noindent
{\bf Remarks.} 1. Theorem \ref{Lr} applies also to operators on $L^r$ which are similar to 
an operator satisfying one of the assumptions (i)-(iv).

2. Kan \cite[Example 3.1]{Kan} has an example (on $\ell^r(\mathbb Z)$) of a positive invertible 
$T$ with $T$ and $T^{-1}$ power-bounded, such that $T$ and $T^{-1}$ are {\it not} Lamperti 
($T^{-1}$ is not positive).
Assani \cite{As} constructed for each $1<r<\infty$ an invertible doubly power-bounded 
operator $T$ on $L^r[0,1]$ which does not satisfy the pointwise ergodic theorem; hence,
by \cite[Theorem 5.3]{Kan}, $T$ is not Lamperti (and none of its powers is).

3. Gillespie \cite[p. 251]{Gi} constructed, for any locally compact $\sigma$-compact Abelian group $G$
and $r \ne 2$, an invertible $T$ on the complex  $L^r(G)$ with $T$ and $T^{-1}$ power-bounded, which is 
{\it not similar} to an isometry (unlike the result for $r=2$ \cite{S-N}), though $T^3$ is an isometry.
We will show below that his construction yields $T$ which in fact is not similar to a Lamperti operator.
\medskip

We recall the structure of Lamperti operators \cite[Theorem 4.1 and Proposition 4.1]{Kan}:
{\it Let $T$ be a Lamperti operator on $L^r(\Omega, \Sigma, \mu)$. Then there exist a 
non-singular $\sigma$-endomorphism $\Phi_0$ of $\Sigma$, which induces a positive operator 
$\Phi$ on measurable functions, and a finite function $h$, such that 
$Tf(\omega) =h(\omega)\cdot \Phi f(\omega)$.} This representation yields
\begin{equation} \label{powers}
T^nf = \big[h\cdot \Phi h\cdots \Phi^{n-1}h\big] \cdot \Phi^n f, \qquad f \in L^r(\Omega,\mu), \quad
n \in\mathbb N.
\end{equation}

For the analysis of the spectrum of invertible Lamperti operators we need the following extension of 
\cite[Theorem 3.1]{Gi}.

\begin{theo} \label{gill1}
Let $1<r<\infty$, $r \ne 2$, and let $T$ be an invertible Lamperti operator on the complex $L^r$ 
with representation $Tf=h\cdot \Phi f$, satisfying $\sup_{n\in\mathbb Z} \|T^n\|<\infty$. 
If $\Phi \ne I$, then there exists a set $\sigma \in \Sigma$ with 
$0<\mu(\sigma) < \infty$ such that $\sigma$ and $\Phi_0(\sigma)$ are disjoint. Hence the set
$$
\mathcal M:=\{1\le m \in \mathbb N:  \exists \sigma_m \in \Sigma, \, 0<\mu(\sigma_m) < \infty, \quad
\sigma_m,\Phi_0(\sigma_m), \dots , \Phi_0^m(\sigma_m) \ \text{are disjoint} \}
$$
is not empty.

(a) If $\mathcal M = \mathbb N$, then $\sigma(T) = \mathbb T$.

(b) If $\mathcal M \ne \mathbb N$, then $\mathcal M$ is finite, with $m_0 \ge 1$ points, 
and $\sigma(T)$ contains $m_0+1$ distinct points equally spaced on $\mathbb T$.
\end{theo}
\begin{proof}
The proof that $\mathcal M$ is not empty is the same as in \cite[Theorem 3.1]{Gi}. 
\smallskip

(a) We adapt the proof of \cite{Gi}. 
If $\mathcal M=\mathbb N$, then for every $m \in \mathbb N$ there is $\sigma_m \in \Sigma$ with  
$1_{\sigma_m}\in L^r$, such that $\big(\Phi_0^k(\sigma_m)\,)_{0\le k\le m} \big)$ are disjoint.
Hence, by (\ref{powers}), $(T^k1_{\sigma_m})_{0\le k \le m}$ have disjoint supports.

Denote $C:=\sup_{n\in \mathbb Z} \|T^n\|$. Then for $f \in L^r$ we have 
\begin{equation} \label{below}
\frac1C \|f\| \le \|T^kf\| \le C\|f\|.
\end{equation}

Fix $\lambda \in \mathbb T$, and put $f_m: = \sum_{n=0}^m \lambda^{-n}T^n 1_{\sigma_m}$.
Disjointness of supports yields
$$
\|f_m\|^r=\sum_{n=0}^m \|T^n 1_{\sigma_m}\|^r \ge 
\frac1{C^r}(m+1)\|1_{\sigma_m}\|^r=\frac{m+1}{C^r} \mu(\sigma_m).
$$
By definition $(\lambda I -T)f_m = \lambda1_{\sigma_m} - \lambda^{-m} T^{m+1}1_{\sigma_m}$, so
$$
\|(\lambda I -T)f_m\| 
\le \|1_{\sigma_m}\| + \|T^{m+1}1_{\sigma_m}\| \le (C+1)\|1_{\sigma_m}\| = (C+1)\mu(\sigma_m)^{1/r}.
$$
Thus 
$$
\frac{\|(\lambda I -T)f_m\| }{\|f_m\|} \le \frac{C(C+1)}{(m+1)^{1/r} } \underset{m\to\infty}\to 0.
$$
Hence $\lambda I-T$ cannot have a bounded inverse, and $\lambda \in \sigma(T)$.
\smallskip

(b) Obviously $n \in \mathcal M$ implies that $m <n$ is in $\mathcal M$, so if 
$\mathcal M \ne \mathbb N$ then $\mathcal M$ is finite, and $\mathcal M=\{1,2,\dots, m_0\}$.
It is proved in \cite[pp. 249-250]{Gi} that there exists $\sigma \in \Sigma$ with 
$0<\mu(\sigma) < \infty$ such that $\sigma,\Phi_0(\sigma), \dots, \Phi_0^{m_0}(\sigma)$ 
are disjoint and $\Phi_0^{m_0+1}(\sigma)=\sigma$.

Let $\Omega_0= \bigcup_{k=0}^{m_0} \Phi_0^k(\sigma)$. Since  for $0 \le k \le m_0$ we have
$$
T1_{\Phi_0^k(\sigma)}= h \cdot\Phi 1_{\Phi_0^k(\sigma)} = h\cdot 1_{\Phi_0^{k+1}(\sigma)}
\quad (\text{addition } \mod m_0),
$$
the subspace $X:=L^r(\Omega_0,\mu)$ is $T$-invariant, and we put $T_0:=T_{|X}$. 
Since $T$ is invertible, it follows form \cite[Proposition 4.1]{Kan} that $\Phi_0(\Omega)=\Omega$,
$\Phi_0$ and its induced operator $\Phi$ are invertible, and $T^{-1}$ is given by
$$
T^{-1}f(\omega) =  \frac1{h(\omega)} \Phi^{-1}f(\omega).
$$
By invertibility, $(\Phi_0^{-1})^k(\sigma) = \Phi_0^{m_0+1-k}(\sigma)$, so we have that $X$
is invariant also under $T^{-1}$, and $T_0$ is invertible, with $T_0^{-1} =(T^{-1})_{|X}$. Hence
$T_0$ and $T_0^{-1}$ are power-bounded, so $\sigma(T_0) \subset \mathbb T$.

Let $\lambda = \e^{2\pi i/(m_0+1)}$ be a primitive  $(m_0+1)$th root of unity, and define 
$S: X \longrightarrow X$, as in \cite{Gi}, by 
$$
Sf:=\sum_{k=0}^{m_0} \lambda^{-k} 1_{\Phi_0^k(\sigma)} f, \quad f \in X=L^r(\Omega_0).
$$
It is computed in \cite{Gi} (with $U$ there replaced by $T_0$) that $S^{-1}T_0S= \lambda T_0$.
so $\sigma(T_0)=\lambda \sigma(T_0)$. Let $\gamma \in \sigma(T_0)$. Then $|\gamma|=1$, and
the $m_0+1$ different points $\gamma, \lambda \gamma, \dots, \lambda^{m_0}\gamma$ are in $\sigma(T_0)$.
Since $\sigma(T_0) \subset \mathbb T$, all its points are in the approximate point spectrum
of $T_0$ \cite[p. 282]{TL}; hence the $m_0+1$ points 
$\gamma, \lambda \gamma, \dots, \lambda^{m_0}\gamma$ are in $\sigma(T)$.
\end{proof}

\begin{lem} \label{spectra}
Let $T$ and $S$ be similar bounded operators on a Banach space $X$. Then $\sigma(T)=\sigma(S)$
and $\sigma_p(T)=\sigma_p(S)$.
\end{lem}
\begin{proof} This is a standard result. A direct computation from $T=V^{-1}SV$ yields 
$R(\lambda,T) =V^{-1}R(\lambda,S)V$ when $\lambda \in \rho(S)$, and $S(Vf)= \lambda Vf$ 
when $Tf =\lambda f$. The lemma then follows.
\end{proof}

\begin{prop} \label{gill}
Let $\mu$ be the normalized Lebesgue measure on the unit circle $\mathbb T$,
and fix $1 < r<\infty$, $r \ne 2$.  Then there exists an invertible operator on the complex
$L^r(\mathbb T,\mu)$ such that both $T$ and $T^{-1}$ are power-bounded, but $T$ is not similar to 
a Lamperti operator (in particular, $T$ is not similar to an isometry), nor to a positive operator.
\end{prop}
\begin{proof}
We will show that Gillespie's construction for $\mathbb T$ \cite[p. 252]{Gi}  (repeated below for 
the sake of completeness) yields $T$ which is not similar to a Lamperti operator, 
nor to a positive one.

\smallskip

Let $\alpha \in (0,1)$ be irrational, and denote by $U_\alpha$ the rotation  of $\mathbb T$
by the angle $2\pi \alpha$. By \cite{Gi75} there exists a bounded operator $A_\alpha$  on 
$L^r(\mathbb T,\mu) $, with $\sigma(A_\alpha) \subset [0,2\pi)$, such that
$U_\alpha=\text{exp}(iA_\alpha)$.  Define $T:= \text{exp}(iA_\alpha/3)$, which is bounded on 
$L^r(\mathbb T,\mu)$. Then we have $\sigma(T) \subset \{\e^{i\theta}: \theta \in [0, 2\pi/3]\}$.
 Putting $\alpha_n = n\alpha -[n\alpha]$, we obtain for $T$ the following representation as 
a multiplier:
$$
T(\sum_n a_n\e^{2\pi in t}) := \sum_n a_n \e^{2\pi i \alpha_n/3}\e^{2\pi i nt}.
$$ 
Clearly $T^3 =U_\alpha$, hence $T$ is  doubly power-bounded, with
$$\sup_{n \in \mathbb Z} \|T^n\| = \max\{1,\|T\|,\|T^2\|\}.$$ 
%

We will use Lemma \ref{spectra} throughout the proof.
\smallskip

Assume that $T$ is similar to a Lamperti operator $Sf =h\cdot \Phi f$.  If $\Phi \ne I$, then 
$\sigma(S) = \sigma(T) \subset \{\e^{i\theta}: \theta \in [0,2\pi/3]\} \ne \mathbb T$ yields a 
contradiction to Theorem \ref{gill1}(b) (as in \cite{Gi}). Thus $\Phi =I$, so $Sf = h\cdot f$. 
By the similarity, $S$ is invertible with $S^{-1}$ power-bounded, so $|h| \equiv 1$ a.e and the 
multiplication operator $S$ is an invertible isometry. We now use an argument implicit in \cite{Gi}. 
Since $S$ is a multiplication operator, so is $S^3$, and therefore $S^3$ is a spectral operator 
of scalar type, by \cite[Theorem 3.1(i)]{Gi}. Hence $U_\alpha$ is similar to a scalar-type 
spectral operator, say $U_\alpha=V^{-1}S^3V$. If $E(\cdot)$ is the spectral measure of $S^3$, 
then $F(\cdot):=V^{-1}E(\cdot)V$ is a spectral measure, and
$$
U_\alpha=V^{-1}S^3V= V^{-1}\big(\int_\mathbb T \lambda \,E(d\lambda)\big)V = 
\int_\mathbb T \lambda\,F(d\lambda),
$$
so $U_\alpha$ is also a spectral operator; but this contradicts Gillespie's 
\cite[Theorem 2(ii)]{Gi75a}. 
Hence $T$ is not similar to a Lamperti operator.

Assume now that $T$ is similar to a positive operator $S$. Then $S$ is  power-bounded, 
and therefore satisfies Lotz's growth condition (G). Hence, by Lotz \cite{Lo} (see 
\cite[Theorem V.4.9, p. 327]{Sch}) $\sigma(S)=\sigma(S)\cap \mathbb T$ is cyclic,  i.e. 
for $\lambda \in \sigma(S)$, all its powers $\lambda^k$ are also in $\sigma(S)=\sigma(T)$. 
But since $\sigma(T)$ is a  proper closed arc of $\mathbb T$, it cannot be cyclic. 
A different proof is by noting that since $L^r$ is reflexive, 
by Gl\"uck \cite[Theorem 5.5]{Gl} the point spectrum $\sigma_p(S)$ is cyclic. 
But the third powers of the eigenvalues of $T$ are $\{\e^{2\pi in\alpha}\}$,
a dense subset of $\mathbb T$, so $\sigma_p(T)=\sigma_p(S)$, being a subset of a proper arc
$\sigma(T)$, cannot be cyclic.
\end{proof}

\noindent{\bf Remarks.} 1. Cuny gave a simple proof that Gillespie's $T$ is not similar to a 
multiplication operator, without using  spectral operators; see Appendix B.

2. Since the cube of the above $T$ is a rotation, $T$ satisfies the pointwise ergodic theorem
without being (similar to) a positive or a Lamperti operator (but $T^3$ is a positive 
Lamperti contraction). 

3. Lemma \ref{spectra} applies also when $S$ is on a different Banach space $Y$, 
and the similarity is by $V$ invertible from $X$ onto $Y$. Thus the proof shows that $T$
defined in Proposition \ref{gill} is not similar to a Lamperti or a positive operator on 
any other $L^r(\Omega',\mu') $ space. Note that $L^r(\Omega,\mu)$ is not isomorphic to
$L^q(\Omega',\mu')$ if $r \ne q$ ($1<r,q<\infty$), since either their type or cotype constants, 
which are invariant under isomorphisms, are different \cite[p. 154, Theorem 6.2.14]{AK} 
(see also \cite[p. 98]{Wo}).

4. Theorem \ref{Lr} does not apply to $T$ of Proposition \ref{gill}, because $T$ 
is not similar to any of the operators in Theorem \ref{Lr}. 
\medskip

We now show that positivity in (iii)  of Theorem \ref{Lr} (which implies the Lamperti property)
is not necessary, so Theorem \ref{ZK} below applies to $T$ of Proposition \ref{gill}. 
The idea of the proof was suggested by Christophe Cuny.

\begin{theo} \label{ZK}
Let $(\Omega,\Sigma,\mu)$ be a $\sigma$-finite measure space and fix $1 < r<\infty$. Let $T$ be
an invertible operator on  (the real or complex) $L^r(\mu)$, such that both $T$ and $T^{-1}$ 
are power-bounded.  Then for every $f \in L^r$ we have norm convergence of 
$\frac1n\sum_{j=1}^n T^{p_j}f$ and of $\frac1N\sum_{j=1}^{\pi(N)}\log p_j T^{p_j}f$.
\end{theo}
\begin{proof}
For the case $r=2$ see Remark 1 to Theorem \ref{prime-average}.

By Proposition \ref{limits} it is enough to prove the 
convergence of $S_nf:=\frac1n\sum_{j=1}^n \Lambda(j)T^jf$ for every $f \in L^r$. We put 
$C:= \sup_{k \in \mathbb Z}\|T^k\|$.

In order to prove that $S_nf$ is Cauchy, it is enough to prove that there is a constant 
$c(f)$ such that for every strictly increasing sequence  $(N_k)$ of positive integers,
\begin{equation} \label{zk}
\sum_{k=1}^\infty \|S_{N_{k+1}}f-S_{N_k}f\|_r^r \le c(f).
\end{equation}
 Fix $(N_k)$; then for $K,M \ge 1$ we have
$$
\sum_{k=1}^K \|S_{N_{k+1}}f-S_{N_k}f\|_r^r \le 
\frac{C^r}M \sum_{m=1}^M \sum_{k=1}^K \|T^m(S_{N_{k+1}}f-S_{N_k}f)\|_r^r =
$$
$$
\frac{C^r}M \sum_{k=1}^K \sum_{m=1}^M 
\int \big| \frac1{N_{k+1}}\sum_{j=1}^{N_{k+1}}\Lambda(j)T^{j+m}f-
\frac1{N_{k}}\sum_{j=1}^{N_{k}}\Lambda(j)T^{j+m}f \big|^r \,d\mu=
$$
$$
\frac{C^r}M \int \sum_{k=1}^K \sum_{m=1}^M 
\big| \frac1{N_{k+1}}\sum_{j=1}^{N_{k+1}}\Lambda(j)T^{j+m}f- 
\frac1{N_{k}}\sum_{j=1}^{N_{k}}\Lambda(j)T^{j+m}f \big|^r \,d\mu  .
$$
We denote the above integrand by
\begin{equation} \label{zk-integrand}
I_{K,M}(\omega):=
\sum_{k=1}^K \sum_{m=1}^M \big| \frac1{N_{k+1}}\sum_{j=1}^{N_{k+1}}\Lambda(j)T^{j+m}f (\omega)- 
\frac1{N_{k}}\sum_{j=1}^{N_{k}}\Lambda(j)T^{j+m}f(\omega) \big|^r  .
\end{equation}

For $\omega \in \Omega$ we define $g:=g_{M,K,\omega}$ on $\mathbb Z$ by $g(m)= T^mf(\omega)$
when $1 \le m \le N_{K+1} +M$ and $g(m) =0$ when $m \le 0$ or $m > N_{K+1}+M$. 
We denote by $\sigma$ the shift on $\mathbb Z$ ($\sigma(n)=n+1$). Then 
$$g(\sigma^j m) = g(m+j)= T^{m+j}f(\omega) \quad \text{for } 1\le m+j \le N_{k+1}+M$$
and $g(\sigma^j m)=0$ otherwise. 

We first prove the theorem when $r>2$.  We estimate the integrand $I_{K,M}(\omega)$ by
$$
I_{K,M}(\omega) =
\sum_{k=1}^K \sum_{m=1}^M \big| \frac1{N_{k+1}}\sum_{j=1}^{N_{k+1}}\Lambda(j)g(\sigma^j m) -
\frac1{N_{k}}\sum_{j=1}^{N_{k}}\Lambda(j)g(\sigma^j m) \big|^r  \le
$$
$$
\sum_{k=1}^K \Big\| \frac1{N_{k+1}}\sum_{j=1}^{N_{k+1}}\Lambda(j) g\circ\sigma^j -
\frac1{N_{k}}\sum_{j=1}^{N_{k}}\Lambda(j)g\circ \sigma^j  \Big\|_{\ell^r(\mathbb Z)}^r .
$$
By the variation estimate of Zorin-Kranich \cite[Theorem 1.1]{ZK} with $q=r > 2$ 
(which holds for complex functions in $\ell^r(\mathbb Z)$), 
there exists a constant $c_r$ such that the last term above is bounded by 
$c_r^r \|g\|_{\ell^r(\mathbb Z)}^r =c_r^r \sum_{m=1}^{N_{K+1}+M} |T^mf(\omega)|^r$. Hence
$$
\sum_{k=1}^K \|S_{N_{k+1}}f-S_{N_k}f\|_r^r \le  
\frac{C^r}M \int c_r^r \sum_{m=1}^{N_{K+1}+M} |T^m f|^rd\mu \le \frac{N_{K+1}+M}MC^{2r}c_r^r\|f\|_r^r.
$$
For $K$ fixed we let $M \to \infty$ and obtain
$\sum_{k=1}^K \|S_{N_{k+1}}f-S_{N_k}f\|_r^r \le  C^{2r}c_r^r \|f\|_r^r$. This proves (\ref{zk})
and yields that $(S_nf)$ is Cauchy.
\medskip

We now prove the theorem when $1<r<2$. For $q>2$ we apply H\"older's inequality  with exponent
$q/r$ to the summation on $k$, and obtain the estimate
$$
I_{K,M}(\omega) =
\sum_{m=1}^M \sum_{k=1}^K \big| \frac1{N_{k+1}}\sum_{j=1}^{N_{k+1}}\Lambda(j)g(\sigma^j m) -
\frac1{N_{k}}\sum_{j=1}^{N_{k}}\Lambda(j)g(\sigma^j m) \big|^r  \le
$$
$$
\sum_{m=1}^M \Big(\sum_{k=1}^K \big| \frac1{N_{k+1}}\sum_{j=1}^{N_{k+1}}\Lambda(j)g(\sigma^j m) -
\frac1{N_{k}}\sum_{j=1}^{N_{k}}\Lambda(j)g(\sigma^j m) \big|^q \Big)^{r/q} K^{(q-r)/q} =
$$
$$
K^{(q-r)/q} \sum_{m=1}^M 
\Big[\Big(\sum_{k=1}^K \big| \frac1{N_{k+1}}\sum_{j=1}^{N_{k+1}}\Lambda(j)g(\sigma^j m) -
\frac1{N_{k}}\sum_{j=1}^{N_{k}}\Lambda(j)g(\sigma^j m) \big|^q \Big)^{1/q} \Big]^r.
$$
The expression in the square brackets is bounded by the $q$-variation norm $\mathcal V_q$
of the sequence $\big(\frac1N\sum_{j=1}^N\Lambda(j) g(\sigma^j(m) \big)_{N>0}$. Hence,
using \cite[Theorem 1.1]{ZK}, we obtain that
$$
I_{K,M}(\omega) \le K^{(q-r)/q} \Big\|\,
\big\|\big(\frac1N\sum_{j=1}^N\Lambda(j) g(\sigma^j(m) \big)_N\big\|_{\mathcal V_q} 
\Big\|^r_{\ell^r(\mathbb Z)} \le 
$$
$$
K^{(q-r)/q} c_{r,q}^r \|g\|_{\ell^r(\mathbb Z)}^r = 
K^{(q-r)/q} c_{r,q}^r\sum_{m=1}^{N_{K+1}+M} |T^mf(\omega)|^r.
$$
Hence
$$
\sum_{k=1}^K \|S_{N_{k+1}}f -S_{N_k}f\|_r^r \le 
\frac{C^r}M \int I_{K,M}(\omega)d\mu \le
\frac{C^r}M c_{r,q}^r K^{(q-r)/q} C^r(N_{K+1}+M)\|f\|_r^r. 
$$
For fixed $K$ we let $M \to \infty$ and obtain 
\begin{equation} \label{cauchy}
\sum_{k=1}^K \|S_{N_{k+1}}f -S_{N_k}f\|_r^r \le {C^{2r}}c_{r,q}^r\|f\|_r^r K^{(q-r)/q}.
\end{equation}
We did not obtain (\ref{zk}), but by the next lemma  $(S_Nf)_N$ is Cauchy.
\end{proof}

\begin{lem}  \label{cuny}
Let $1\le r < \infty$ and let $(f_n)_{n\ge 1}$ be a sequence in a normed space  $X$. 
If there exist $L>0$ and $\delta \in (0,1)$ such that for every increasing subsequence $(n_j)$ 
and every $K \in \mathbb N$ we have
$$
\sum_{j=1}^K\|f_{n_{j+1}}-f_{n_j}\|^r \le L\cdot K^\delta
$$
then $(f_n)$ is Cauchy in $X$.
\end{lem}
\begin{proof} If $(f_n)$ is not Cauchy, then for some $\epsilon>0$ we can construct an 
increasing sequence $(n_j)$ with $\|f_{n_{2j}} -f_{n_{2j-1}}\| > \epsilon$. Then 
for any $K$ we obtain $K\epsilon^r \le L(2K)^\delta$, which yields a contradiction.
\end{proof}

\noindent
{\bf Remarks.} 1. Lemma \ref{cuny} and its use for proving Theorem \ref{ZK} when $r<2$
are inspired by the approach of Bourgain in \cite[p. 209]{B2} (see also \cite[p. 220]{B2}). 
Bourgain's estimates were used in \cite{JOW}.


2. Guy Cohen noted that the proof of Theorem \ref{ZK} is valid for non-invertible power-bounded
$T$ if for some $C'$ we have $\|f\|_r \le C' \|T^mf\|_r$ for every $f \in L^r$ 
(e.g. $T$ is any isometry), since invertibility is used only in the estimate
$\|S_{N_{k+1}}f-S_{N_k}f \|_r \le C\|T^m(S_{N_{k+1}}f-S_{N_k}f)\|_r$. In fact, the proof is 
valid when $T$ is a power-bounded quasi-isometry, provided we take $M$ from the sequence $(M_n)$
in the definition of quasi-isometry. It follows that the Lamperti condition of Theorem \ref{Lr}(iv) 
is not necessary when $T$ there is power-bounded.

3. Proposition \ref{gill} provides examples to which Theorem \ref{ZK} applies while Theorem 
\ref{Lr} does not.
\bigskip

\section{Problems} 

In this section we list some problems which arise from our results.
\medskip

1. {\it Is there a necessary and sufficient condition on a sequence $(a_n)_{n\in \mathbb N}$
which ensures that $\|\frac1n \sum_{k=1}^n a_kT^k x\| \to 0$ for every weakly almost periodic $T$
and every flight vector of $T$?} In particular, is $(a_n) \in \overline{\mathcal A}^{W_1}$
necessary? The sequences $\Lambda(n))$ and $\Lambda'(n))$ are not in $\overline{\mathcal A}^{W_1}$,
but we do not know whether they modulate all flight vectors of weakly almost periodic operators;
they modulate all almost periodic operators, by \cite[Proposition 1.4]{LOT}.
\smallskip

2. {\it Is Proposition \ref{stable} true for power-bounded operators in $H$?}
The answer is probably "NO", since the Blum-Hanson theorem fails, by M\"uller-Tomilov \cite{MT}.
\smallskip


3. {\it If $\sup_n \frac1n\sum_{k=1}^n |a_k| < \infty$ and $c(1):=\lim \frac1n \sum_{k=1}^n a_k$ 
exists, does $\ba$ modulate all weakly stable operators in $H$ (extending Corollary 
\ref{stable+fix})?} It might be true, with negative answer to the previous problem.
\smallskip

4. {\it Is Theorem \ref{prime-average}(ii) true also for power-bounded operators in $H$?}
In view of the recent positive result of ter Elst and M\"uller \cite{EM},  proving convergence
of the averages along the squares, the cubes  and other polynomials, the answer might be "YES".
\smallskip

5. {\it Is Theorem \ref{Lr} true also for power-bounded positive Lamperti operators?} The proof
of Theorem \ref{Lr} uses the pointwise convergence of averages along the primes, proved in 
\cite{JOW} for operators satisfying (iv) of the theorem. For power-bounded positive Lamperti
operators, without invertibility,  we are looking for the norm convergence only.

\bigskip

\bigskip

\section{APPENDIX A: On the maximal inequality along the primes } 
\smallskip

\centerline{\it by M\'at\'e Wierdl}
\smallskip

The proof of the pointwise convergence of averages along the primes  in \cite{W} is based on
transforming the problem to the equivalent one of convergence of modulated averages with the 
modified von Mangoldt weights $(\Lambda'(k))$. The next lemma is the abstraction of this equivalence.

\begin{lem} \label{w1}
Let $(t_k)_{k>0}$ be a sequence of reals. Then $\frac1n\sum_{j=1}^n t_{p_j}$ converges 
if and only if $\frac1N\sum_{j=1}^{\pi(N)} t_{p_j}\log p_j$ converges.
\end{lem}
\begin{proof} Lemma 1 of \cite{W} is valid when then sequence $f(T^ax)$ is replaced by a sequence 
$(t_a)$, with the same proof. Then use the remark following the statement of \cite[Lemma 1]{W}.
\end{proof}

In \cite{W} a maximal inequality  was proved for the $\Lambda'$-modulated averages.
We show below that this inequality yields (is equivalent to) the maximal inequality for 
the averages along the primes.

\begin{lem} \label{w2}
There exists a constant $c_1$ such that for any sequence of reals $(t_k)_{k>0}$ we have
$$
\sup_N \frac1N\sum_{j=1}^{\pi(N)} |t_{p_j}|\log p_j \le
c_1 \sup_n \frac1n\sum_{j=1}^n |t_{p_j}| .
$$
\end{lem}
\begin{proof} Since $\frac{\pi(N)\log N}N \to 1$, we put $c_1 := \sup_N \frac{\pi(N)\log N}N < \infty $.
Since $p_{\pi(N)}\le N$, for $(t_k)$ non-negative we have
$$
\frac 1N \sum_{j=1}^{\pi(N)} t_{p_j}\log {p_j} \le \frac 1N \sum_{j=1}^{\pi(N)} t_{p_j}\log N = 
\frac{\pi(N)\log N}N \frac1{\pi(N)}\sum_{j=1}^{\pi(N)} t_{p_j} 
\le c_1 \sup_n \frac1n\sum_{j=1}^n t_{p_j}\ ,
$$
which yields the desired inequality. 
\end{proof}

\begin{lem} \label{w3}
There exist constants $c$ and $c_2$ such that for any sequence
$(t_k)$ such that $t^*:= \sup_n \frac1n \sum_{k=1}^n |t_k| < \infty$ we have 
\begin{equation} \label{sup-estimate}
 \sup_n \frac1n\sum_{j=1}^n |t_{p_j}| \le 
c_2 t^* + c \sup_N \frac1N\sum_{j=1}^{\pi(N)} |t_{p_j}|\log p_j \,.
\end{equation}
\end{lem}
\begin{proof} 
Define $c :=2 \sup_{N\ge 2} \frac N{\pi(N)\log N}$  and $c_2:= \sup \frac{\sqrt N}{\pi(N)}$. 
Denote $\mathbb P_{N}:= \{p\in\mathbb P, p\le N\}$.  We may assume $t_k \ge 0$.  Then 
$$
\frac1{\pi(N)} \sum_{p\in\mathbb P_{\sqrt N}} t_p \le 
\frac1{\pi(N)} \sum_{k=1}^{[\sqrt N]} t_k=
\frac{[\sqrt N]}{\pi(N)}\cdot \frac1{[\sqrt N]} \sum_{k=1}^{[\sqrt N]} t_k 
\le c_2 t^*
$$ 
and for $N >1$ we have 
$$
\frac1{\pi(N)} \sum_{\sqrt N < p\in \mathbb P_{N}} t_p \le
\frac1{\pi(N)} \sum_{ \sqrt N < p\in \mathbb P_{N}}\frac{\log p}{\log \sqrt N }\ t_p =
$$
$$
\frac N{\pi(N)\frac12\log N} \cdot \frac1N \sum_{p\in \mathbb P_{N}} (\log p)t_p \le
c \sup_N \frac1N \sum_{j=1}^{\pi(N)}t_{p_j}\log {p_j} \ .
$$
Putting the above together we obtain (\ref{sup-estimate}).
\end{proof}

\begin{cor}
Let $(\Omega,\Sigma,\mu)$ be a $\sigma$-finite measure space and fix $1<r<\infty$. Let $T$
be a mean-bounded (e.g. power-bounded) operator on $L^r(\mu)$. If 
\begin{equation} \label{max-p}
\big\| \sup_n |\frac1n \sum_{j=1}^n T^{p_j}f| \, \big\|_{L^r} \le C_1 \|f\|_{L^r} 
\quad \forall f\in L^r
\end{equation}
then also
\begin{equation} \label{max-L}
\big\| \sup_n |\frac1N \sum_{j=1}^{\pi(N)} \log p_j T^{p_j}f| \big\|_{L^r} \le 
C_2 \|f\|_{L^r} \quad \forall f\in L^r.
\end{equation}
Conversely, if $T$ satisfies (\ref{max-L}) and 
\begin{equation} \label{max-e}
\big\| \sup_n |\frac1n \sum_{k=1}^n T^kf| \, \big\|_{L^r} \le C_0 \|f\|_{L^r} 
\quad \forall f\in L^r
\end{equation}
then (\ref{max-p}) holds.
\end{cor}

The above corollary applies to $T$ induced by a measure preserving 
transformation, in particular to $T$ induced on $\ell^r(\mathbb Z)$ or $\ell^r(\mathbb Z^+)$ 
by the shift. It applies also to Lamperti quasi-isometries on $L^r$ (see Proposition \ref{kan}).
\medskip

Recently, Pavel Zorin-Kranich \cite{ZK} proved an estimate for the $q$-variation of the 
modulated averages of the shift in $\ell^r(\mathbb Z)$ along the primes, with  modulation
by the von Mangoldt function $\Lambda(k)$.  Similar estimates when the modulation is by 
$\Lambda'(k)$ are given by Mirek, Trojan and Zorin-Kranich \cite{MTZ}. These inequalities
are equivalent to an inequality for the $q$-variation of the averages along the primes 
without modulation. The proof will be published elsewhere.
\bigskip

\section{APPENDIX B: On certain multiplier operators in $L^p$, $p \ne 2$}
\smallskip
\centerline{\it by Christophe Cuny}
\medskip

Let $\mu$ be Lebesgue's measure on $[0,1)$, and identify $\mathbb T$ with $[0,1)$ by $t \to \e^{2\pi it}$.
For $n\in \Z$, set $e_n(t)={\rm e}^{2i\pi nt}$, $t\in [0,1)$. 

Fix $2 \ne p \in (1,\infty)$ and let $(a_n)_{n\in \Z}$ be a sequence of unimodular complex numbers 
which is a {\it multiplier} in $L^p$,
i.e. the operator 
\begin{equation} \label{multip}
T \big(\sum_{n\in \Z}c_n(f) e_n\big) =\sum_{n\in \Z}a_n c_n(f)e_n
\end{equation}
defines an invertible doubly power-bounded operator on $L^p$.

\begin{prop}
If $2\ne p \in(1,\infty)$ and the $a_n$ are all different, then the multiplier operator $T$  defined
on $L^p$ by (\ref{multip}) is not similar to a multiplication operator $Sf=h\cdot f$.
\end{prop}
\begin{proof}
By contradiction, assume that there exist invertible operators $S$ and $V$ such that $Sf=hf$ and 
$S=VTV^{-1}$. 
\smallskip

For $n\in \Z$, set $f_n:=Ve_n$ and put $A_n:= support (f_n)$. Then
$$
Sf_n=VTe_n=a_n f_n=hf_n , \quad n \in \N .
$$
Hence $h \equiv a_n$ on $A_n$ ($\mu$-almost everywhere).  Since $a_n\neq a_m$ for $m\neq n\in \Z$  
by assumption, $\mu(A_n\cap A_m)=0$. The sets $(A_n)_{n\in \Z}$ are therefore disjoint (modulo $\mu$). 
It follows that 
\begin{gather*}
\|f_1+\ldots +f_n\|_p^p =\sum_{i=1}^n \|f_i\|_p^p\, .
\end{gather*}
For $C=\max\{\|V\|^p,\|V^{-1}\|^p \}$ we have 
\begin{gather*}
\|e_1+\ldots +e_n\|_p^p/C\le \|f_1+\ldots +f_n\|_p^p \le C \|e_1+\ldots +e_n\|_p^p\, ,
\end{gather*}
and
\begin{gather*}
n/C= \sum_{i=1}^n \|e_i\|_p^p/C\le \sum_{i=1}^n \|f_i\|_p^p\le C\sum_{i=1}^n \|e_i\|_p^p=Cn\, .
\end{gather*}
We obtain  from the above that
$  n/C^2 \le \|e_1+\ldots +e_n\|_p^p\le C^2n$, which yields a contradiction, since
$\|e_1+\ldots +e_n\|_p^p\ge \|e_1+\ldots +e_n\|_2^p=n^{p/2}$ when $p>2$, and 
\newline\noindent
$\|e_1+\ldots +e_n\|_p^p\le \|e_1+\ldots +e_n\|_2^p=n^{p/2}$  when $1<p<2$. 
\end{proof}

\bigskip

{\bf Acknowlegements.} The authors are grateful to Guy Cohen, Christophe Cuny,
 M\'at\'e Wierdl and Manfred Wolff for several helpful discussions.

\end{document}